\documentclass[10pt]{amsart}

\usepackage{amsfonts}
\usepackage{amssymb}
\usepackage{mathrsfs}
\usepackage{amsthm}
\usepackage{amsmath}
\usepackage{graphics}
\usepackage{color}

\theoremstyle{plain}
\newtheorem{thm}{Theorem}[section]
\newtheorem{thmr*}{Theorem}[section]
\newtheorem{prop}[thm]{Proposition}

\newtheorem{lemma}[thm]{Lemma}
\newtheorem{cor}[thm]{Corollary}

\newtheorem{remark}[thm]{Remark}
\newtheorem{defn}[thm]{Definition}
\newcommand{\bb}[1]{\mathbb{#1}}
\newcommand{\cl}[1]{\mathcal{#1}}

\begin{document}

\title{Some New Equivalences of Anderson's Paving Conjectures}
\author[V. I. Paulsen]{Vern I. Paulsen}
\address{Dept. of Mathematics, University of Houston, Houston, TX, 77204}
\email{vern@math.uh.edu}

\author[M. Raghupathi]{Mrinal Raghupathi}
\address{Dept. of Mathematics, University of Houston, Houston, TX, 77204}
\email{mrinal@math.uh.edu}
\date{\today}
\thanks{This research was supported in part by NSF grant
  DMS-0600191. Portions of this research were begun while the first
  author was a guest of the American Institute of Mathematics.}

\begin{abstract}
Anderson's paving conjectures are known to be equivalent to the
Kadison-Singer problem. We prove some new equivalences of Anderson's
conjectures that require the paving of smaller sets of matrices. We
prove that if the strictly upper triangular operatorss are pavable, then every
0 diagonal operator is pavable. This result follows from a new paving
condition for positive operators. In addition, we prove that if the
upper triangular Toeplitz operators are paveable, then all Toeplitz
operators are paveable.
\end{abstract}

\maketitle

\section{Introduction}
Anderson\cite{An1, An2, An3} proved various "paving" conjectures that he showed are equivalent to the Kadison-Singer conjecture\cite{KS} about whether or not pure states on a discrete MASA in B(H) extend uniquely to states on B(H).  Recently, some research has focused on finding smaller sets of operators for which it is sufficient to establish Anderson's paving conjecture. In particular, the question arose as to if paving was known to hold for all strictly uppertriangular operators, then could it be shown to hold for all operators?
It turns out that the answer to this last question is affirmative as
was first shown in \cite{CT}. We are able to present a much shorter
proof of this fact. In fact, the result follows as a fairly easy consequence of recent work of Blecher and Labuschange \cite{BL} on boundary representations for logmodular algebras.  Perhaps even more interesting, is that the key element of the proof of the theorem of \cite{BL} is a technique of Hoffman \cite{KH2} for proving uniqueness of representing measures on uniform algebras. Hoffman's idea readily adapts to states on C*-algebras, which in turn leads to some new paving results.
In the process of proving these new paving conditions, we provide our
own proofs of some of Anderson's results.

\section{Paving Results}

Throughout this section we fix real numbers, $0 < a < 1 < b$ and given a unital C*-algebra $\cl B$, we let $\cl P[a,b]$ denote the closed, convex set of positive elements of $\cl B$ such that $aI \le P \le bI.$
The following result is based on an idea of Hoffman\cite{KH2}.

\begin{thm} \label{u1} Let $\cl B$ be a unital C*-algebra and let $s_i: \cl B \to \bb C, i=1,2$ be states. Then the following are equivalent:
\begin{itemize}
\item[(i)] $s_1=s_2,$
\item[(ii)] for every positive, invertible $p \in \cl B, s_1(p)s_2(p^{-1}) \ge 1,$
\item[(iii)] for every $p \in \cl P[a,b], s_1(p)s_2(p^{-1}) \ge 1.$
\end{itemize}
\end{thm}
\begin{proof} To see that (i) implies (ii), it is sufficient that if $s$ is a state and $q$ is positive and invertible, then $s(q)s(q^{-1}) \ge 1.$ To see this, note that for any real number $t, 0 \le s((tq + q^{-1})^2) = t^2s(q^2) +2 + s(q^{-2})$ and thus this quadratic has no roots or a repeated root, from which 
$1 \le s(q^2)s(q^{-2})$ follows. Now choose, $q^2=p.$

Clearly, (ii) implies (iii). To see that (iii) implies (i), let $h=h^* \in \cl B$, so that for $t$ real and near 0, $e^{th} \in \cl P[a,b]$  Hence, $f(t)= s_1(e^{th})s_2(e^{-th}) \ge 1$ for $t$ in some neighborhood of 0. Since, $f(0)=0,$ we have that $0=f^{\prime}(0) = s_1(h)- s_2(h).$ Since $h=h^*$ was arbitrary, $s_1=s_2.$
\end{proof}

Let $\cl S \subseteq \cl B$ be an {\bf operator system}, i.e., a
subspace that contains the identity and satisfies, $X \in \cl S$
implies $X^* \in \cl S.$ Given a state on $\cl S$, i.e., a unital, positive, linear functional $s: \cl S \to \bb C,$ and $h=h^* \in \cl B,$ we define
$$\ell_s(h) = \sup \{ s(k) : k \le h, k \in \cl S \}$$
and
$$u_s(h) = \inf \{ s(k) : h \le k, k \in \cl S \}.$$

\begin{prop}\label{ext} Let $\cl B$ be a unital C*-algebra, let $\cl S \subseteq
  \cl B$ be an operator system, let $s: \cl S \to \bb C$ be a state
  and let $h=h^* \in \cl B.$ Then for every $t, \ell_s(h) \le t \le
  u_s(h),$ there exists a state, $s_t:\cl B \to \bb C$ extending $s$
  such that $s_t(h) = t.$
\end{prop}
\begin{proof} We may assume that $h \notin \cl S.$ Let $\cl T$ be the operator system spanned by $\cl S$
  and $h$, i.e., $\cl T = \{ a+ \lambda h: a \in \cl S, \lambda \in
  \bb C \}$ and define $f: \cl T \to \bb C$ by $f(a+ \lambda h) = s(a)
  + \lambda t.$

Note that if $a + \lambda h \ge 0,$ then $a=a^*$ and $\lambda \in \bb
R.$
If $\lambda >0,$ then $h \ge - \lambda^{-1}a$ and hence, $t \ge
\ell_s(h) \ge s(-\lambda^{-1}a),$ from which it follows that $f(a+ \lambda h) \ge 0.$
Similarly, if $\lambda <0,$ then $- \lambda^{-1}a \ge h,$ and
$s(-\lambda^{-1}a) \ge u_s(h) \ge t,$ from which it follows that $f(a+
\lambda h) \ge 0.$

Thus, $f$ is a state on $\cl T.$ But a state on an operator system is a contractive linear
functional and hence by the Hahn-Banach theorem $f$ can be extended to a
contractive linear functional $s_t$ on $\cl B$. But since $s_t$ is
unital and contractive, it is a state.
\end{proof}

\begin{thm} \label{u2} Let $\cl B$ be a unital C*-algebra, let $\cl S \subseteq \cl B$ be an operator
  system and let $s: \cl S \to \bb C$ be a state. Then the following are equivalent:
\begin{itemize}
\item[(i)] $s$ extends uniquely to a state on $\cl B,$
\item[(ii)] for every $h=h^* \in \cl B, \ell_s(h) = u_s(h),$
\item[(iii)] for every positive invertible $p \in \cl B, \ell_s(p) \ell_s(p^{-1}) \ge 1,$
\item[(iv)] for every $p \in \cl P[a,b], \ell_s(p) \ell_s(p^{-1}) \ge 1.$
\end{itemize}
\end{thm}
\begin{proof} The equivalence of (i) and (ii) follows from the above proposition.

Now assuming (i), if we let $s_1: \cl B \to \bb C$ denote the unique state extension of $s$, then by the above proof, necessarily, $s_1(h) = \ell_s(h),$ and hence, $\ell_s(p) \ell_s(p^{-1}) = s_1(p) s_1(p^{-1}) \ge 1.$
Thus, (i) implies (iii).

Clearly, (iii) implies (iv). Assuming (iv), if $s_1,s_2$ are any two state extensions of $s$, then $s_1(p) s_2(p^{-1}) \ge \ell_s(p) \ell_s(p^{-1}) \ge 1$ for all $p \in \cl P[a,b],$ and hence $s_1=s_2,$ and, thus, (i) follows.
\end{proof}

\begin{defn} Given a unital C*-algebra $\cl B,$ an operator system
  $\cl S \subseteq \cl B$ and a state $s: \cl S \to \bb C$ we let $\cl
  U(s) = \{ b \in \cl B: s_1(b)=s_2(b) \}$ where $s_1,s_2$  are
  arbitrary states extending $s.$ We call this set the {\bf uniqueness domain for s.}
\end{defn}

\begin{remark}
It is not hard to see that $\cl U(s)$ is an operator system and that by \ref{ext}, $h=h^* \in \cl U(s)$ if and only if $\ell_s(h) = u_s(h).$
However, given a single positive, invertible $p \in \cl B$, it is not clear if 
 $\ell_s(p)\ell_s(p^{-1}) \ge 1,$ implies that $p \in \cl U(s).$ The above proof does show that if every positive, invertible $q$ in the unital C*-algebra generated by $p$ satisfies $\ell_s(q) \ell_s(q^{-1}) \ge 1,$ then the entire C*-algebra generated by $p$ is contained in $\cl U(s).$
In this sense, the condition $\ell_s(p) \ell_s(p^{-1}) \ge 1,$ is a weaker condition.
\end{remark}

We remark also that \ref{ext} shows that the interval, $[\ell_s(h), u_s(h)]$ is exactly the range of all possible images of $h$ attained by extensions of $s$.
In this sense it is the {\bf interval of non-uniqueness.}

We now turn to the situation of the Kadison-Singer conjecture. To this
end, we let $\cl B= B(\ell^2(\bb N))$ and identify operators $X \in
\cl B$ with their infinite matrices, $X=(x_{i,j}).$ We let $\cl D$
denote the MASA of operators that are diagonal with respect to the
canonical orthonormal basis for $\ell^2(\bb N)$ and let $E:
B(\ell^2(\bb N)) \to \cl D$ denote the conditional expectation onto $\cl D,$ given by $E((x_{i,j}))=(d_{i,j})$ where $d_{i,i}=x_{i,i}$ and $d_{i,j} =0, i \ne j.$ 

We shall freely identify $\cl D$ with the continuous functions on the Stone-Cech compactification of the natural numbers, $\beta \bb N.$ 
In particular, if $A \subseteq \bb N$ we shall let $P_A=(p_{i,j})$
denote the diagonal projection with $p_{i,i}=1$ if and only if $i \in
A.$ Such a projection is identified with the characteristic function
of the closure of $A$ in $\beta \bb N,$ which is a clopen set. 

We shall also make use of the one-to-one, onto correspondence between points in $\beta \bb N$ and {\bf ultrafilters} on $\beta \bb N.$ To recall this correspondence, note that since $\bb N$ is dense in $\beta \bb N,$ every clopen set $U$ is uniquely determined by $U \cap \bb N.$ Given $\omega \in \beta \bb N,$ the collection of subsets of $\bb N$ given by
$$\frak U(\omega) = \{ U \cap \bb N: \omega \in U \},$$
where $U$ denotes an arbitrary clopen neighborhood of $\omega$ is an ultrafilter on $\bb N.$

\begin{lemma} Let $\cl H$ and $\cl K$ be Hilbert spaces and let $H=
  \begin{pmatrix} A & B\\ B^* & C \end{pmatrix} \in B(\cl H \oplus \cl
  K)$ be self-adjoint with $A$ positive and invertible. Then there exists, $\delta >0$ such that $H + \delta P_{\cl K} \ge 0,$ where $P_{\cl K}$ denotes the orthogonal projection onto $\cl K.$
\end{lemma}
\begin{proof} Let $X= A^{-1/2}B,$ then \begin{multline*}\langle
  \begin{pmatrix} A & B\\ B^* & C+\delta I_{\cl K} \end{pmatrix} \begin{pmatrix} h\\k
    \end{pmatrix}, \begin{pmatrix} h\\k \end{pmatrix}\rangle = \\ \langle
    Ah,h \rangle + \langle A^{1/2}Xk,h \rangle + \langle X^*A^{1/2}h,k
    \rangle + \langle Ck,k \rangle + \delta \|k\|^2  \ge \\
\|A^{1/2}h \|^2 - 2 \|Xk\| \|A^{1/2}h \| - \|C\| \|k\|^2 + \delta
\|k\|^2 \ge \\ (\|A^{1/2}h\| - \|Xk\|)^2 + (\delta - \|C\| -
\|X\|^2)\|k\|^2 \ge 0 \end{multline*} provided that we choose $\delta \ge \|C\| +
\|X\|^2.$
\end{proof}

\begin{thm} \label{ui} Let $\omega \in \beta \bb N,$ let $s_{\omega}: \cl D \to \bb C$ be the *-homomorphism given by evaluation at $\omega,$ and let $H=H^* \in B(\ell^2(\bb N)).$ Then $\ell_{s_{\omega}}(H) = u_{s_{\omega}}(H)=t$ if and only if for every $\epsilon > 0$ there exists $A \in \frak U(\omega)$ such that
$(t - \epsilon)P_{A} \le P_AHP_A \le (t+ \epsilon)P_A.$
\end{thm}
\begin{proof} If $s$ is any state that extends $s_{\omega},$ then $s(P_A)=1$ and so $s(P_AXP_A)=s(X).$ Thus, if the second condition holds, then $t - \epsilon \le s(P_AHP_A) \le t+ \epsilon$ and hence, $s(H)=t$ for every state extension. Thus, by \ref{ext}, the first condition holds.

Conversely, if the first condition holds, then given $\epsilon >0$, there exists $D_1,D_2 \in \cl D$ with $D_1 \le H \le D_2$ such that $t - \epsilon \le s_{\omega}(D_1) \le s_{\omega}(D_2) \le t+ \epsilon.$ Thus, we may find a neighborhood $U$ of $\omega$ such that the functions $D_i$ are, respectively, greater than $t - \epsilon$ and less than $t + \epsilon$ on $U$. Let $A= U \cap \bb N,$ so that $A \in \frak U(\omega).$ Then by the lemma we may choose values, $\delta_1, \delta_2$ so that $(t- \epsilon)P_A + \delta_1(I-P_A) \le D_1 \le H \le D_2 \le (t+\epsilon)P_A + \delta_2(I-P_A),$ and the result follows by pre and post multiplying this inequality by $P_A$.
\end{proof}

The equivalence of (i) and (v) below, is originally due to
Anderson\cite{An1, An2} and is the basis of his paving results. The proof
that we give shares some key elements with his proof, but we feel is sufficiently different to merit inclusion.

\begin{thm} \label{u3} Let $\omega \in \beta \bb N,$ and let $s_{\omega}: \cl D \to \bb C$ be the *-homomorphism given by evaluation at $\omega.$ Then the following are equivalent:
\begin{itemize}
\item[(i)] $s_{\omega}$ extends uniquely to a state on $B(\ell^2(\bb N)),$
\item[(ii)] for every $H=H^* \in B(\ell^2(\bb N))$ with $E(H)=0, \ell_{s_{\omega}}(H) = 0,$
\item[(iii)] for every positive, invertible $P \in B(\ell^2(\bb N)), \ell_{s_{\omega}}(P) \ell_{s_{\omega}}(P^{-1}) \ge 1,$
\item[(iv)] for every $P \in \cl P[a,b], \ell_{s_{\omega}}(P) \ell_{s_{\omega}}(P^{-1}) \ge 1,$
\item[(v)] for each $H=H^* \in B(\ell^2(\bb N))$ with $E(H)=0,$ and each $\epsilon > 0,$ there exists $A \in \frak U(\omega)$ with $- \epsilon P_A \le P_AHP_A \le + \epsilon P_A,$
\item[(vi)] for each positive, invertible $P \in B(\ell^2(\bb N))$ and each $\epsilon >0,$ there exists $A \in \frak U(\omega)$ and real numbers $c,d >0,$ with $1 - \epsilon < cd,$such that $cP_A \le P_APP_A$ and $dP_A \le P_AP^{-1}P_A,$
\item[(vii)] for each $P \in \cl P[a,b],$ and each $\epsilon > 0,$ there exists $A \in \frak U(\omega)$ and real numbers $c,d > 0,$ with $1 - \epsilon < cd,$ such that $cP_A \le P_APP_A$ and $dP_A \le P_AP^{-1}P_A.$
\end{itemize}
\end{thm}
\begin{proof}
The equivalence of (i), (iii) and (iv), follows from the equivalence of (i), (iii) and (iv) in \ref{u2}.

Moreover, condition (ii) above is easily seen to be equivalent to condition (ii) in \ref{u2}, by applying the new condition (ii) to $H- E(H)$ and $E(H) - H.$

We now prove the equivalence of (iii) and (vi). The proof of the equivalence of (ii) with (v) and of (iv) with (vii) is identical.

First to see that (iii) implies (vi), given $\epsilon >0,$ we may
choose $D_1, D_2 \in \cl D$ such that $D_1 \le P, D_2 \le P^{-1}$ and
$1- \epsilon < s_{\omega}(D_1) s_{\omega}(D_2).$ From this it follows
that we may pick $c,d >0$ with $1 - \epsilon < cd$ and a clopen set
$U$ that is a neighborhood of $\omega,$ such that the continuous
functions $D_1$ and $D_2$ are strictly greater than $c$ and $d$, respectively,
on $U.$ For a sufficiently large negative number, $n,$ we will have
that $D_1 -[c \chi_U + n(I - \chi_U)]$ and $D_2 -[d \chi_U + n(I-
\chi_U)],$ are positive and invertible, where $\chi_U$ denotes the characteristic
function of the set $U.$ Let $A= U \cap \bb N,$ so that $P_A= \chi_U,$
then $cP_A= P_A(cP_A + n(I- P_A))P_A \le P_APP_A$ and $dP_A \le
P_A(dP_A + n(I -P_A))P_A \le P_AP^{-1}P_A,$  with $P_APP_A - cP_A$ and
$P_AP^{-1}P_A$ both
positive and invertible. Hence,  (vi) follows by applying the above lemma.

Conversely, assuming (vi), and slightly perturbing $c$ and $d$, if
necessary, we may assume that $P_APP_A-cP_A \ge \delta P_A$ and
$P_AP^{-1}P_A -dP_A \ge \delta P_A$ for some $\delta > 0.$ Hence, we
by applying the lemma twice, we may pick a
sufficiently large negative number, $n$, so that $cP_A + n(I-P_A) \le
P$ and $dP_A + n(I-P_A) \le P^{-1}.$ Hence, $c= s_{\omega}(cP_A +n(I -P_A)) \le \ell_{s_{\omega}}(P),$ and $d= s_{\omega}(dP_A + n(I-P_A)) \le \ell_{s_{\omega}}(P^{-1}),$ from which it follows that $\ell_{s_{\omega}}(P) \ell_{s_{\omega}}(P^{-1}) > 1- \epsilon$ and hence, we have (iii).
\end{proof}

It is customary to say that {\bf the Kadison-Singer conjecture is true} if every pure state on $\cl D,$ i.e., if every state of the form $s_{\omega}$, extends uniquely to $B(\ell^2(\bb N)).$ This can be a bit misleading, since Kadsion and Singer never actually made this conjecture and there is some indication that they might have believed the negation of this statement.

We shall call a finite collection of disjoint subsets, $\{A_1,...A_r \}$ with $\bb N= A_1 \cup ... \cup A_r$ an {\bf r-paving of $\bb N.$} 

The equivalence of (i), (ii) and (iii) below is also in Anderson\cite{An3}. We include them for completeness and because our proof of the equivalence of (i) and (ii) is slightly different.

\begin{thm} \label{ks} The following are equivalent:
\begin{itemize}
\item[(i)] the Kadison-Singer conjecture is true,

\item[(ii)] for each $H= H^* \in B(\ell^2(\bb N))$ with $E(H)=0,$ and each $\epsilon >0,$ there exists an $r$ and an $r$-paving $\{A_1,...,A_r \}$ of $\bb N$ with $ - \epsilon P_{A_i} \le P_{A_i}HP_{A_i} \le + \epsilon P_{A_i},$

\item[(iii)] for each $\epsilon >0,$ there exists an $r$ such that for every $H= H^* \in B(\ell^2(\bb N))$ with $E(H)=0,$ then there exists an $r$-paving $\{A_1,...,A_r \}$ of $\bb N$ with $\| P_{A_i}HP_{A_i} \| \le + \epsilon \|H\|, i=1,...,r,$

\item[(iv)] for each positive invertible, $P \in B(\ell^2(\bb N))$ and each $\epsilon >0,$ there exists an $r,$ an $r$-paving $\{A_1,...,A_r \}$ of $\bb N$ and positive real numbers, $c_1,...c_r, d_1,...,d_r$ with $c_id_i > 1 - \epsilon,$ such that $c_iP_{A_i} \le P_{A_i}PP_{A_i}$ and $d_iP_{A_i} \le P_{A_i}P^{-1}P_{A_i}$ for $i=1,...,r,$

\item[(v)] for each positive invertible, $P \in \cl P[a,b]$ and each $\epsilon >0,$ there exists an $r,$ an $r$-paving $\{A_1,...,A_r \}$ of $\bb N$ and positive real numbers, $c_1,...c_r, d_1,...,d_r$ with $c_id_i > 1 - \epsilon,$ such that $c_iP_{A_i} \le P_{A_i}PP_{A_i}$ and $d_iP_{A_i} \le P_{A_i}P^{-1}P_{A_i}$ for $i=1,...,r,$

\item[(vi)] for each $\epsilon >0,$ there exists an $r$, such that for every $P \in \cl P[a,b]$ there is an $r$-paving $\{A_1,...,A_r \}$ and positive real numbers, $c_1,...,c_r,d_1,...,d_r,$ with $c_id_i > 1 - \epsilon$ such that $c_iP_{A_i} \le P_{A_i}PP_{A_i}$ and $d_iP_{A_i} \le P_{A_i}P^{-1}P_{A_i},$ for $i= 1,...,r.$

\end{itemize}
\end{thm}
\begin{proof} The proofs of the equivalence of (i) with each of  (ii), (iv)
  and (v) are essentially the same. One notes that by \ref{u3}, uniqueness of the extension for each $\omega$ yields a set $A_{\omega}$ which corresponds to a clopen neighborhood $U_{\omega}$ of $\omega$ in $\beta \bb N.$ But $\beta \bb N$ is compact so that some finite subcollection $\{U_1,...U_t \}$ of these sets covers $\beta \bb N$ and consequently, $B_i= U_i \cap \bb N$ covers $\bb N.$ Now let $\{A_1,...A_r \}$ denote the minimal non-empty elements of the finite Boolean algebra of sets generated by the $B_i$'s.

To see the uniformity of $r$ in $\epsilon,$ for (iii) and (vi), first note that in (iii), by scaling it is sufficient to consider $\|H\|=1.$ If one assumes, as in Anderson's proof, that there is no upper bound on $r$, then one takes a sequence(either of $H_n=H_n^*, E(H_n)=0, \|H_n\|=1$ in the case of (iii) or of positives, $P_n \in \cl P[a,b]$ in (vi)) with corresponding $r$'s tending to infinity and gets a contradiction by considering the operator that is their direct sum. 

\end{proof}

The uniformity in the dependence of $r$ on $\epsilon$ is the main advantage of restricting to the smaller set of positive invertibles, $\cl P[a,b].$

\begin{remark} If  a single operator $H=H^*$ satisfies (ii) or (iii), then every pure
  state on the diagonal extends uniquely to $H,$ that is, $H \in  \cap_{\omega \in \beta \bb N} \cl U(s_{\omega}).$  Also, it can be shown
  that if a positive invertible $P$ has the property that $H= P -E(P)$
  satisfies either (ii) or (iii), then $P$ satisfies (iv). However, it
  is not clear that if a single $P$ satisfies (iv), then $P-E(P)$ satisfies
  (ii) or (iii).  For this reason, we believe that (iv) and (v) might
  be "easier" conditions to verify, if indeed, Kadison-Singer is true. See also the remark following \ref{u2}. It is important to note that because
  of the exponentiation trick in the heart of \ref{u1}, having the
  condition met for a single positive tells us nothing about
  uniqueness of extension for that single positive operator, unlike
  the situation for self-adjoints.
  \end{remark}
  
It is possible, but somewhat tedious, to give a direct "paving" proof that (v) implies (iii), so in this sense (v) might not lead to results that couldn't have been seen directly through "classical" paving arguments. To see how to accomplish this, one first starts with an arbitrary projection, $Q$ and sets $P= aQ + b(I-Q)$ and uses (v) to derive some paving estimates for $Q$. The argument then proceeds by using spectral projections for $H$.  

We now turn our attention to the results on paving upper triangular
matrices mentioned in the introduction.
To this end we call an operator $T=(t_{i,j}) \in B(\ell^2(\bb N))$
{\bf upper triangular} provided that $t_{i,j} =0$ for all $i >j$ and
we let $\cl T(\bb N)$ denote the unital subalgebra of upper triangular
operators.
We call an operator {\bf strictly upper triangular} if $t_{i,j}=0$ for
all $i \ge j,$ and let $\cl T_0(\bb N)$ denote the subalgebra of
strictly upper triangular. Note that $T \in \cl T_)(\bb N)$ if and
only if $T \in \cl T(\bb N)$ and $E(T)=0.$

It is well-known \cite{Ar} that $\cl T_0(\bb N) + \cl T_0(\bb N)^*$ is not dense
in $\{B \in B(\ell^2(\bb N)): E(B)=0 \},$ since triangular
truncation is unbounded. This fact makes the following
results somewhat surprising. 

\begin{thm} Let $\omega \in \beta \bb N,$ and let $s_{\omega}: \cl D
  \to \bb C$ be the *-homomorphism given by evaluation at $\omega.$
  Then the following are equivalent:
\begin{itemize}
\item[(i)] $s_{\omega}$ extends uniquely to a state on $B(\ell^2(\bb
  N)),$
\item[(ii)] for every $T \in \cl T_0(\bb N), \ell_{s_{\omega}}(T+T^*)=0,$
\item[(iii)] for each $T \in \cl T_0(\bb N),$ and each $\epsilon >0,$
  there exists $A \in \frak U(\omega),$ such that $-\epsilon P_A \le
  P_A(T+T^*)P_A \le + \epsilon P_A,$
\item[(iv)] for each $T \in \cl T_0(\bb N)$ and each $\epsilon >0,$
  there exists $A \in \frak U(\omega),$ such that $\|P_ATP_A\| <
  \epsilon.$
\end{itemize}
\end{thm}
\begin{proof} Clearly, (iv) implies (iii) implies (ii).

We now prove that (ii) implies (i). Let $s_1,s_2$ be two states on
$B(\ell^2(\bb N))$ that
extend $s_{\omega}.$ Assuming (ii), we have that $s_1(T+T^*)= s_1((iT)
+ (iT)^*) =0,$
for every $T \in \cl T_0(\bb N)$ and, hence, $s_1(T)=0.$ Similarly,
$s_2(T)=0$ for every $T \in \cl T_0(\bb N).$
Hence, for every $T \in \cl T(\bb N),$ we have that $s_1(T)=s_2(T) =
s_{\omega}(E(T)).$

Now let $P \in B(\ell^2(\bb N)),$ be positive and invertible. Then,
since the upper triangulars are a logmodular subalgebra of
$B(\ell^2(\bb N))$ \cite{Ar}, we
may factor $P=T^*T$ with $T \in \cl T(\bb N)$ and invertible.

Hence, $s_1(P)s_2(P^{-1}) = s_1(T^*T)s_2(T^{-1}T^{*-1}) \ge
|s_1(T)s_2(T^{-1})|^2 =
|s_{\omega}(E(T))s_{\omega}(E(T^{-1}))|^2
= |s_{\omega}(E(T)E(T^{-1}))|^2= |s_{\omega}(I)|^2= 1,$ where the last
equality follows since $E: \cl
T(\bb N) \to \cl D$ is a unital homomorphism.

Thus, by \ref{u1}, $s_{\omega}$ extends uniquely.

Finally, assuming (i), fix $\epsilon >0$ and let $T \in \cl T_0(\bb N).$ Applying the
equivalence of (i) and (v) in \ref{u3}, we get sets $A_1,A_2 \in \frak
U(\omega)$ such that $- \epsilon P_{A_1} \le P_{A_1}(T+T^*)P_{A_1} \le +\epsilon P_{A_1}$ and $- \epsilon P_{A_2} \le P_{A_2}((iT) + (iT)^*)P_{A_2} \le +\epsilon P_{A_2}.$ Hence, if we let $A= A_1 \cap A_2$, then $A \in \frak U(\omega)$
and $\|P_ATP_A \| \le \epsilon.$
\end{proof}

We now obtain the results of \cite{CT}.

\begin{cor} The following are equivalent:
\begin{itemize}
\item[(i)] the Kadison-Singer conjecture is true,
\item[(ii)] for each $T \in \cl T_0(\bb N)$ and each $\epsilon >0,$ there exists an $r$-paving $\{A_1,...,A_r \}$ of $\bb N$ with $- \epsilon P_{A_i} \le P_{A_i}(T+T^*)P_{A_i} \le + \epsilon P_{A_i},i=1,...,r,$
\item[(iii)] for each $T \in \cl T_0(\bb N)$ and each $\epsilon >0,$ there exists an $r$-paving $\{A_1,...A_r \}$ of $\bb N$ such that $\|P_{A_i}TP_{A_i} \| < \epsilon, i=1,...,r.$
\end{itemize} 
\end{cor}

\begin{cor} The following are equivalent:
\begin{itemize}
\item[(i)] for each $\epsilon >0$ there exists $r_1$ such that if $H=H^*$ and $E(H)=0$ then there exists an $r_1$-paving $\{A_1,..., A_{r_1} \}$ of $\bb N$ such that $\|P_{A_i}HP_{A_i}\| \le \epsilon \|H\|, i=1,...,r_1,$
\item[(ii)] for each $\epsilon >0$ there exists $r_2$ such that if $T \in \cl T_0(\bb N),$ then there exists an $r_2$-paving $\{A_1,...,A_{r_2} \}$ of $\bb N$ such that $\|P_{A_i}TP_{A_i}\| \le \epsilon \|T\|, i=1,...,r_2.$
\end{itemize}
\end{cor}

Because of the non-constructive nature of our earlier proofs, there is no clear bound on $r_1$ in terms of $r_2.$ Clearly, $r_2 \le 2r_1,$ by considering $T+T^*$ and $(iT)+(iT)^*$ as above.

Analogous results apply to Toeplitz operators. Recall that if $f \in L^{\infty}(\bb T),$ where $\bb T$ denotes the unit circle in the complex plane and we set $\hat{f}(n) = \frac{1}{2\pi}\int_0^1 f(e^{it}) e^{-2\pi int} dt$, then by the {\bf Toeplitz operator with symbol f} we mean the operator, $T_f \in B(\ell^2(\bb N))$ whose matrix is given by $T_f= (\hat{f}(i-j)).$
We identify $H^{\infty}(\bb D) = H^{\infty}(\bb T)$ with the subspace
of $L^{\infty}(\bb T)$ such that $\hat{f}(n)=0$ for all $n<0$ and
$H^{\infty}_0(\bb T)$ with those functions such that $\hat{f}(n) =0$
for all $n \le 0.$ As with operators, we have that $H^{\infty}(\bb T)
+ \overline{H^{\infty}(\bb T)}$ is not dense in $L^{\infty}(\bb T),$
indeed, the result for operators follows from this fact.
By a classic factorization theorem\cite{KH2}, every positive invertible function $p \in L^{\infty}(\bb T)$ is of the form $p= |f|^2$ for some invertible $f \in H^{\infty}(\bb T)$ and hence, $T_p = T_f^*T_f$ with $T_f$ invertible.

\begin{thm} The following are equivalent:
\begin{itemize}
\item[(i)] for each $\epsilon >0$ there exists $r_1$ such that if $T_h$ is a self-adjoint Toeplitz operator with $\hat{h}(0)=0,$ then there is an $r_1$-paving $\{A_1,...,A_r\}$ of $\bb N$ such that $\|P_{A_i}T_hP_{A_i} \| \le \epsilon \|T_h\|, i=1,...,r_1$
\item[(ii)] for each $\epsilon >0$ there exists $r_2$ such that if $T_f$ is a Toeplitz operator with $f \in H^{\infty}_0(\bb T),$ then there is an $r_2$-paving $\{A_1,...,A_{r_2} \}$ of $\bb N$ such that $\|P_{A_i}T_fP_{A_i}\| \le \epsilon \|T_f\|, i=1,...,r_2.$
\end{itemize}
\end{thm}
\begin{proof}  Applying the first condition to $T_{Re(f)}$ and $T_{Im(f)}$, clearly yields the second condition with $r_2 \le 2r_1.$

Conversely, the second condition is equivalent to every $s_{\omega}$ extending uniquely to the Toeplitz operators with symbol in $H^{\infty}_0(\bb T).$  Now fix $\omega$ and let $s_i: B(\ell^2(\bb N)) \to \bb C, i=1,2$ be two states that extend $s_{\omega}.$ Then  $\rho_i: L^{\infty}(\bb T) \to \bb C,i=1,2$ defined by $\rho_i(f) = s_i(T_f), i=1,2$ are states on the C*-algebra $L^{\infty}(\bb T)$ that are both the homomorphism, $\rho_1(f)=\rho_2(f) = \hat{f}(0)$ on $H^{\infty}(\bb T).$ Thus, by either applying the theorem of \cite{BL} on uniqueness of extensions of boundary representations of logmodular algebras or observing that $\rho_1(p)\rho_2(p^{-1}) \ge 1$ and applying \ref{u1}, we obtain that $\rho_1= \rho_2$ and hence every state $s_{\omega}$ extends uniquely to the Toeplitz operators. But this is equivalent to this family of operators being uniformly pavable.
\end{proof}


\begin{thebibliography}{10}

\bibitem{An1} J. Anderson, {\it Extensions, restrictions and representations of states on C*-algebras,} Trans. Amer. Math. Soc. {\bf 249}(1979), 303-329.

\bibitem{An2} J. Anderson, {\it Extreme points in sets of linear maps on B(H),} J. Func. Anal. {\bf 31}(1979), 195-217.

\bibitem{An3} J. Anderson, {\it A conjecture concerning pure states of B(H) and related theorems,} in Proceedings, Vth International Conference Operator Algebras, Timisoara and Herculane, Romania, Pitman, New York/London, 1984.

\bibitem{Ar} W.B. Arveson, {\it Interpolation problems in nest
    algebras,} J. Func. Anal. {\bf 20}(1975), 208-233.

\bibitem{BL} D. Blecher and L. Labuschange, {\it Logmodularity and
    isometries of operator algebras,} Trans. Amer. Math. Soc. {\bf
    355}(2002), 1621-1646.

\bibitem{CT} P. Casazza and J. Tremain, {\em The paving conjecture is
    equivalent to the paving conjecture for triangular matrices,} preprint.


\bibitem{KH1} Kenneth Hoffman, \textit{Banach Spaces of Analytic Functions,} Dover, New York, (1988).

\bibitem{KH2} Kenneth Hoffman, \textit{Analytic Functions and
    Logmodular Banach Algebras}, Acta. Math., \textbf{108}, (1962),
  271-317.

\bibitem{KS} R. Kadison and I. Singer, {\it Extensions of pure
    states,} Amer. J. Math. {\bf 81}(1959), 547-564.



\end{thebibliography}
\end{document}